\documentclass[11pt,a4paper]{article}

\usepackage[T1]{fontenc}
\usepackage[utf8]{inputenc}
\usepackage[english]{babel}

\usepackage{amsmath}
\usepackage{amssymb}
\usepackage{amsthm}
\usepackage{mathrsfs}

\usepackage{xcolor}
\usepackage{booktabs}
\usepackage{tabularx}
\usepackage[shortlabels]{enumitem}
\usepackage{graphicx}
\usepackage{authblk}          
\usepackage[margin=2.5cm]{geometry}
\usepackage[hidelinks]{hyperref}

\providecommand{\keywords}[1]{%
  \par\medskip\noindent\textbf{Keywords.} #1\par}
\providecommand{\support}[1]{%
  \begingroup\renewcommand{\thefootnote}{}\footnotetext{#1}\endgroup}

\theoremstyle{plain}
\newtheorem{theorem}{Theorem}[section]
\newtheorem{lemma}[theorem]{Lemma}
\newtheorem{proposition}[theorem]{Proposition}
\newtheorem{corollary}[theorem]{Corollary}
\theoremstyle{definition}
\newtheorem{defi}[theorem]{Definition}
\newtheorem{example}[theorem]{Example}
\theoremstyle{remark}
\newtheorem{remark}[theorem]{Remark}

\newcommand{\verde}[1]{\textcolor{black}{#1}}
\newcommand{\azul}[1]{\textcolor{black}{#1}}
\definecolor{darkgreen}{RGB}{0,100,0}

\newcommand{\F}{\ensuremath{\mathbb{F}}}%

\newcommand\C{\mathrm{C}}

\DeclareMathOperator{\lcm}{lcm}

\title{Factorization patterns and fields of definition of $\ell$-torsion points on the Jacobians of genus 3 
hyperelliptic curves}

\author{Amalia Pizarro-Madariaga}
\affil{Instituto de Matemáticas, Universidad del Valpara\'iso, Chile}

\author{Edgardo Riquelme}
\affil{Departamento de Ciencias Básicas, Universidad del B\'io-B\'io, Chile}

\begin{document}

\maketitle

\begin{abstract}

Motivated by Schoof–Pila–type point-counting algorithms, we study the fields of definition and factorization patterns of Galois orbits of $\ell$-torsion points on the Jacobians of genus-3 hyperelliptic curves over finite fields.

We show that the degree of the field of definition of the $\ell$-torsion points can be bounded by $O(\ell^4)$, improving the previously expected $O(\ell^6)$ bound   (which reflects the size of the $\ell$-torsion subgroup, of order $\ell^6$ for a genus 3 Jacobian). Moreover, we establish a precise correspondence between the rank of the $\ell$-torsion subgroup and the Galois orbits of $\ell$-torsion divisors. Our approach relies on a detailed analysis of the Jordan decomposition of the Frobenius action on $J$ and its nilpotent part.

\end{abstract}

\support{This work was partially supported STICAmSud
project 24-STIC-04. Author $^2$ was supported in part by grant 2020433 IF/R (Universidad
del B\'io-B\'io, Chile) }

\section{Introduction}
The discovery of Schoof's algorithm \cite{rs} generated significant interest in the efficient computation of zeta functions of curves defined over finite fields, due to applications in cryptographic curve construction.\\
Gaudry et al. \cite{GS00, GS04, GS12} showed that Schoof's elliptic curve point-counting algorithm \cite{rs} can be effectively adapted to hyperelliptic curves of genus two defined over a prime field, up to cryptographic sizes. (The first adaptation to the very general case of an abelian variety can be found in the work of Pila in 1990 \cite{Pila}.) This makes it possible to compute the order of the group associated with the \azul{hyperelliptic curve}, and thereby to select curves suitable for cryptographic use.\\
\azul{The idea behind Schoof-type algorithms is to study the action of the Frobenius endomorphism on the $\ell$-torsion subgroup $J[\ell]$ of the Jacobian of a hyperelliptic curve in order to determine the characteristic polynomial modulo $\ell$. Repeating this process for several primes $\ell$ and combining the resulting congruences via the \verde{ chinese remainder theorem} allows one to recover the full characteristic polynomial.
\\
In Schoof's algorithm for elliptic curves, division polynomials \verde{vanish precisely on the abscissa of $\ell$-torsion points}.  In genus 2, generalized division polynomials \cite{Cantor1994} do not provide such a description; moreover, no equally simple representation of the $\ell$-torsion subgroup is known, and constructing such a representation becomes the main computational bottleneck. This phenomenon also appears to persist for curves of genus $g>2$ \cite{Abelard2018} .
\\
To accelerate Schoof-type algorithms, Gaudry and Schost exploited $\ell^k$-torsion elements. Since the cost of computing $\ell$-torsion grows rapidly with $\ell$, it is often more efficient to work with $\ell^k$-torsion for small primes (such as $\ell=2,3,5,7$) rather than using larger primes $\ell$.\\
To compute the $\ell^k$-torsion subgroup, one recursively computes successive $\ell$-divisions starting from the $\ell$-torsion subgroup itself. An $\ell$-division consists of algebraically inverting the multiplication-by-$\ell$ map  on the Jacobian variety of $C$, $[\ell]:J\to J.$ 
More precisely, given a divisor $D_\ell$, the goal is to describe algebraically all divisors $D$ satisfying \verde{$[\ell]D\thicksim  D_{\ell}$} (see  \cite{GS12}, \cite{Miret24}, \cite{riquelme2016},\cite{Riquelme2019}.)
}

\subsection{Motivation}
Work of Abelard, Gaudry and Spaenlehauer~\cite{abelard2019counting}  on genus-3 point-counting algorithms has continued this line of research. Furthermore, recent improvements to $\ell$-division algorithms~\cite{Miret24} have impacted Schoof–Pila-type algorithms, raising the question of how much the complexity of general algorithms \verde{(i.e. without efficiently computable real multiplication)} in genus 3 decreases, and making it important to revisit the possibility of the generic case \verde{(i.e. without assuming real multiplication)}.
In  Schoof–Pila–type algorithms, it is essential to construct the $\ell$-torsion subgroup. The analogue of the division polynomials for hyperelliptic curves is given by Cantor \cite{Cantor1994}. Other approaches have been suggested by Eid \cite{EID202368} using $p$-adic differential equations. In the specific case $g=3$, Lupoian in \cite{LUPOIAN26} gives a practical method for computing the 3-torsion subgroup of the Jacobian of a hyperelliptic curve.\\
The importance of determining the field of definition of the \(\ell\)-torsion subgroup appears in the following two situations. Since torsion points are usually obtained through the factorization of auxiliary polynomials over finite fields, the complexity of these algorithms strongly depends on the corresponding extension degree. As noted by Matsuo \cite{IM10}, the distinct-degree factorization (DDF) step in the Cantor-Zassenhaus-type \cite{cantor1981} algorithm dominates the complexity, so knowledge of the field of definition can significantly improve the efficiency of these algorithms. Another application arises in the construction  of towers of field extensions required in Schoof-type algorithms to compute \(\ell\)-division points \cite{GS12, Miret24}.
Such explicit algorithms highlight the importance of understanding the arithmetic structure of the $\ell$-torsion subgroups. In general, the possible factorization types of the Galois orbits of the $\ell$-torsion divisors are determined by the characteristic polynomial $P(x)$ of the reduced Frobenius endomorphism $\pi$ modulo $\ell$.
\subsection{Literature overview and main result}
For elliptic curves (genus 1), this was studied by Verdure \cite{verdure}. 
For hyperelliptic curves of genus 2, the number of distinct cases to consider increases significantly. An upper bound analysis for irreducible factors can be found in Matsuo \cite{IM10}, and an application to the factorization types of $\ell$-modular polynomials can be found in Gaudry and Schost \cite{GS05}. Furthermore, Riquelme \cite{riquelme2016} has established the relationship between the factorization type of the Galois orbits of the $\ell$-torsion divisors and the field of definition of the $\ell$-divisions.
To our knowledge, a genus 3 version of polynomial factorization patterns has not yet been developed. On the other hand, bounds for genus 3 could be interesting considering applicability \azul{(concrete algorithms for computing the order of cryptographically secure genus-3 curves)}.
 This work generalizes Gaudry’s approach in genus 2 by providing a classification of the possible Frobenius action matrices on the $\ell$-torsion of Jacobians of genus-3 hyperelliptic curves. As the genus increases, genuinely new configurations appear, and several arguments valid in lower genus no longer apply directly. This classification constitutes a first step toward establishing effective bounds for the fields of definition of $\ell$-torsion points.\\
We derive explicit upper bounds for the degree of the field of definition of the $\ell$-torsion, proving in particular that this degree is bounded by $O(\ell^4)$. More precisely, 
\verde{\begin{theorem}{(Theorem \ref{teo1} in Section 4)}
 If $ \ell > 2$, there exists a positive integer \(d \leq \ell^4 - \ell^3 + \ell - 1\) such that
\[
J[\ell] \subset J(\mathbb{F}_{q^d}).
\] 
\end{theorem}}
\verde{Additionally, we provide an example  showing that the bound of Theorem 1.1 is attained for $\ell=3$.} 
\subsection{Methodology}
 We begin by classifying the characteristic polynomial of the Frobenius endomorphism in genus 3, determining which cases are compatible with the Hasse–Weil theorem and the Weil pairing. Next, using the elementary divisor version of the rational canonical form theorem, we derive a general framework that generalizes the factorization patterns to genus 3. Rather than presenting an exhaustive list of all possible factorizations, we adopt a more abstract and conceptual approach; nevertheless, we include detailed analyses of the case $\ell=2$.
Finally, in contrast to Matsuo’s approach, which relies on an exhaustive case-by-case analysis, we focus on the worst-case contribution of the nilpotent part in the Jordan decomposition, resulting in uniform upper bounds.\\
 The paper is organized as follows. In Section 2, we recall the necessary background on Jacobians of genus-3 hyperelliptic curves over finite fields, the action of Frobenius on the $\ell$-torsion subgroup, and the relation between the characteristic polynomial modulo $\ell$ and Galois orbits. In Section 3, we analyze the possible factorization patterns of the $\ell$-torsion in genus 3, based on the Jordan decomposition of the Frobenius endomorphism, and relate these patterns to fields of definition of $\ell$-torsion points.
In Section 4, we derive explicit upper bounds for the degree of the field of definition of the $\ell$-torsion, and prove the main result showing that this degree is bounded by $O(\ell^4)$.

\section{Background}

 Let \( \mathbb{F}_q \) be a finite field where $q=p^k$ with $p>2$ prime and \( \overline{\mathbb{F}}_q \) be an algebraic closure of \( \mathbb{F}_q \).
A hyperelliptic curve of genus \( 3 \) over \( \mathbb{F}_q \) is \verde {given by an} equation \( C :y^2=f(x)\), where $f(x)\in \mathbb{F}_q[x]$ is a monic polynomial of degree $7$ with distinct roots.

We denote by \( J = \mathrm{Jac}(C) \) the Jacobian variety of $C$. For any field extension $K$ of $\mathbb{F}_q$, we denote by $J(K)$ the abelian group of $K$-rational points of $J$. 

For a prime number \( \ell\),  the \( \ell \)-torsion subgroup of the Jacobian is defined by

$$ J[\ell]=\{D\in J(\overline{\mathbb{F}}_q): [\ell]D=0\}, $$
 where $0$ denotes the identity element of $J(\overline{\mathbb{F}}_q)$ and $[\ell]$ represents the multiplication-by-$\ell$ map.

It is known that if \( \ell\neq p\), then
$J[\ell]\simeq (\mathbb{Z}/\ell\mathbb{Z})^6$.
Let $\pi(x,y)=(x^q,y^q)$ be the Frobenius endomorphism acting on $C$. This induces an endomorphism of the Jacobian, which we also call the Frobenius endomorphism and denote by $\pi$. This endomorphism induces a linear action on \( J[\ell] \),
which we represent by a matrix \( T\in M_{6\times 6}( \mathbb{F}_{\ell} )\).

If $T_{\ell}(J)$ is the Tate module of $J$ and $\rho$ is an endomorphism of $J$, we denote by $T_{\ell}(\rho)$ the corresponding
element in $\operatorname{End}_{\mathbb{Z}_{\ell}}(J)$.
\begin{defi}
We define the characteristic polynomial of the Frobenius  as the characteristic polynomial of $T_{\ell}(\pi)$
$$P(X):=\det(XI_d-T_{\ell}(\pi))$$
\end{defi}
The polynomial $P(X)$ has degree 6, has integer coefficients, and is independent of $\ell$ \cite[Theorem 4.83]{2005ehcc}. As a consequence of the Weil conjectures, it can be written as
\begin{equation*}
  P(X) = X^{6} - s_{1}X^{5} + s_{2}X^{4} - s_{3}X^{3} + s_{2}qX^{2} - s_{1}q^{2}X + q^{3}.  
\end{equation*}
For details, see \cite[Proposition 8.4]{2005ehcc}. We denote by

\begin{equation*}
  \widetilde{P}(X) = X^{6} - \widetilde{s}_{1}X^{5} + \widetilde{s}_{2}X^{4} - \widetilde{s}_{3}X^{3} + \widetilde{s}_{2}qX^{2} - \widetilde{s}_{1}\widetilde{q}^{2}X + \widetilde{q}^{3}      
\end{equation*}
the reduction of $P(X)$ modulo $\ell$. The roots of $P(X)$ satisfy the following relations
\begin{theorem}\cite[Theorem 5.76]{2005ehcc}\label{roots}
Let \(\alpha_1,\ldots,\alpha_6\) be the roots of $P$. 
\begin{itemize}
    \item[(i)] Each $\alpha_i$ is an algebraic integer of degree $\leqslant 6$.
    \item[(ii)] We can numerate the roots such that for $1 \leqslant i \leqslant 3$ we have
    \[
    \alpha_i \alpha_{i+3} = q.
    \]
    \item[(iii)] For $1 \leqslant i \leqslant 6$ take any embedding of $\alpha_i$ into $\mathbb{C}$. Then the complex absolute value $|\alpha_i|$ is equal to $\sqrt{q}$.
\end{itemize}

\end{theorem}
For details of the proof, see \cite[Theorem 5.1.15]{Stichtenoth1993}.

We will introduce the following notation:
\begin{itemize}
    \item $[n_1]^{c_1}\cdots[n_k]^{c_k}$ denotes a squarefree polynomial having $c_i$ irreducible factors of degree $n_i$ in $\mathbb{F}_{\ell}[X]$.
    \item $(n_1)^{c_1}\cdots(n_k)^{c_k}$ describes the factorization pattern of the Galois orbit of $\ell$-torsion points. Each $n_i$ represents the degree of an irreducible factor, which corresponds to the size of a Galois orbit, and the exponent $c_i$ denotes the number of distinct orbits of size $n_i$.
    
\end{itemize}

The factorization of $P(X)$ \verde{into} irreducible factors $[n_1]^{c_1}\cdots[n_k]^{c_k}$ over $\mathbb{F}_{\ell}$ determines the structure of $J[\ell]$. In fact, the decomposition of $J[\ell]$ into disjoint Galois orbits is determined by the factorization pattern of the characteristic polynomial $\widetilde{P}(X)$. Each irreducible factor of degree $n$ implies the existence of a Galois orbit with points defined over an extension of degree $n$. 

The degrees $n_i$ correspond to the degree of the extensions $\mathbb{F}_{\ell^n}$ which contain the eigenvalues of the matrix $T$. More precisely, if  $\mathbb{F}_{\ell^n}$ is the minimal extension of $\mathbb{F}_{\ell}$ that contains all eigenvalues of $T$, it can be decomposed as
\[
T = P A P^{-1},
\]
where $A \in GL_6(\mathbb{F}_{\ell^n})$ is in Jordan canonical form and 
$P \in GL_6(\mathbb{F}_{\ell^n})$. For the details, see \cite[Theorem 8.6]{roman2008}.

\begin{defi}
A square matrix $A$ is called nilpotent if there exists a positive integer $k$ such that $A^k = 0$. The smallest such integer $k$ is called the index of nilpotence of $A$.
\end{defi}

\begin{remark}
 It is easy to see that if $A\in M_{n\times n}$  is strictly upper triangular, then $A^n = 0$. In fact, the characteristic polynomial of $A$ is $p(x)=x^n$ and by the Cayley-Hamilton theorem, $p(A)=A^n=0.$  
\end{remark}

\section{$\ell$-torsion pattern factorization}
For hyperelliptic curves of genus 3, the possible factorization types of the characteristic polynomial of Frobenius over $\mathbb{Z}$ are very restricted. Up to isogeny, the genuinely \azul{new} cases are the irreducible type  [6] and the completely split type  $[2][2][2]$;
all other factorization types correspond to Jacobians that are isogenous to products of Jacobians of genus-1 or genus-2 curves.

\begin{theorem}[{\cite[Theorem 1]{Xing}}]
The possible characteristic polynomials of abelian varieties of dimension three over $\mathbb{F}_{q}$ (where $q=p^k$) are all $f(t) \in \mathbb{Z}[t]$ of degree six satisfying one of the following conditions:
\begin{enumerate}[(i)]
    \item $f(t)=t^{6}+a_{1}t^{5}+a_{2}t^{4}+a_{3}t^{3}+a_{2}qt^{2}+a_{1}q^{2}t+q^{3}$ whose roots are Weil numbers, $v_p(d(0))/k \in \mathbb{Z}$ for any irreducible factor $d(t)$ of $f(t)$ in $\mathbb{Q}_{p}[t]$, and the polynomial $t^{3}+a_{1}t^{2}+(a_{2}-3q)t+a_{3}-2qa_{1}$ is irreducible over $\mathbb{Q}$.
    
    \item $f(t)=(t^{2}+\beta t+q)^{3}$, where $\beta$ is equal to either:
    \begin{enumerate}[(a)]
        \item $\beta=aq^{1/3}$ with $3 \mid k$, $p \nmid a$, and $|a| < 2q^{1/6}$.
        \item $\beta=\pm q^{2/3}$ with $q=8$ or $27$.
    \end{enumerate}
    
    \item $f(t)=f_{1}(t)f_{2}(t)$, where $f_{1}(t)$ and $f_{2}(t)$ are the characteristic polynomials of abelian varieties of dimensions one and two, respectively.
\end{enumerate}
The abelian varieties in case (i) and case (ii) are simple, and not simple in case (iii). The abelian varieties in case (ii) are supersingular.
\end{theorem}

 When reducing modulo $\ell$, however, these types may give rise to additional phenomena, such as repeated factors and nontrivial nilpotent parts in the Frobenius action. For this reason, the case 
[2][2][2] plays a central role in genus 3 and accounts for the most significant new situations in our analysis.
\subsection{Classification of the characteristic polynomial of Frobenius endomorphism}
 We analyse all possible reductions of the Frobenius endomorphism matrix and derive the corresponding $\ell$-factorization patterns. This classification is based on the possible factorizations of the minimal and characteristic polynomials of the Frobenius endomorphism $\pi$.
\begin{lemma}
If the characteristic polynomial of the Frobenius endomorphism has a quintuple root modulo \(\ell\),
then the multiplicity is actually \(6\).
\end{lemma}
\begin{proof}
 Let  $\widetilde{P}(X)=(X-A)^5(X-B)$ be the characteristic polynomial of $\pi$ modulo $\ell$. By Weil's theorem, if $\lambda$ is a root of the characteristic polynomial of $\pi$, then $q/\lambda$ is also a
root of it. Therefore we have $A^2 \equiv q \bmod{\ell}$ and $AB \equiv q \bmod{\ell}$ . But $q$ is nonzero
modulo $\ell$, so that $A \equiv B \bmod{\ell}$.
\end{proof} 

\begin{lemma}
Suppose that $\tilde{P}(X)\in\mathbb{F}_{\ell}[X] $ has an irreducible factor of degree 3. Then its  factorization pattern is necessarily [3][3].
\end{lemma}
\begin{proof}
Let $\alpha,\beta,\gamma\in\mathbb{F}_{\ell^3}\setminus\mathbb{F}_{\ell}$ be the roots of the irreducible cubic factor of $\tilde{P}(X)$. By  Theorem \ref{roots} there exists a root, say $\gamma$, such that $\gamma=\tilde{q}/\gamma'$ where $\gamma'$ is a root of the remaining factors. If $\gamma'$ were a root of a quadratic or linear polynomial, it would imply that $\gamma$ also belongs to $\mathbb{F}_{\ell^2}$, which is a contradiction. Therefore, $\gamma'\in\mathbb{F}_{\ell^3}\setminus\mathbb{F}_{\ell}$ and the remaining factor must be an irreducible cubic polynomial.  
\end{proof}

\begin{lemma}
Suppose that $\tilde{P}(X)\in\mathbb{F}_{\ell}[X] $ has linear irreducible factors of the form $(X-\alpha)^3(X-\beta)^2$, for some nonzero $\alpha,\beta\in\mathbb{F}_{\ell}$. Then its factorization pattern is necessarily  $[1]^4[1]^2$ or $[1]^3[1]^3$.
\end{lemma}
\begin{proof}
It is enough to observe that if $(X-\gamma)$ is the remaining linear factor, then $q/\gamma$ is equal to $\alpha$ or $\beta$. 
\end{proof}

The Jordan form of the Frobenius endomorphism on $J[\ell]$ is strongly constrained by the arithmetic and symplectic structure of the Jacobian. In particular, when $\pi$ has a unique eigenvalue in $\mathbb{F}_\ell$, not every non-semisimple Jordan configuration can occur. The following lemma excludes several highly non-diagonalizable cases in genus $3$.

\begin{lemma}
Let $C$ be a hyperelliptic curve of genus $3$ defined over a finite field $\mathbb{F}_q$, and let $\pi$ be the Frobenius endomorphism acting on $J[\ell]$, where $\ell \nmid q$. Suppose that $\pi$ has a unique eigenvalue $A\in \mathbb{F}_\ell$.
Then the Jordan
form of $\pi$ cannot be  $[1]^5[1],$   $[1]^3[1]^2[1]$  or  $[1]^3[1][1][1]$ .
\end{lemma}
\begin{proof}
Suppose that the Jordan form of $\pi$ is  $[1]^5[1]$. Let $V\subset\mathbb{F}_{\ell}^6$ be a
$\pi$-stable subspace of dimension 5. If $A$ is the only eigenvalue of $\pi$ in $\mathbb{F}_\ell$, then there exists a basis  $\{D_1,D_2,D_3,D_4,D_5\}$ of $V$ such that the action of $\pi$ on a basis $\{D_1,D_2,D_3,D_4,D_5,D_6\}$ of $\mathbb{F}_{\ell}^6$ is given by the following matrix:
$$ \begin{bmatrix}
A & 1 & 0 & 0 & 0 & 0 \\
0 & A & 1 & 0 & 0 & 0 \\
0 & 0 & A & 1 & 0 & 0 \\
0 & 0 & 0 & A & 1 & 0 \\
0 & 0 & 0 & 0 & A & 0\\
0 & 0 & 0 & 0 & 0 & A
\end{bmatrix}.$$ 
Following the idea in \cite{GS05}, we can use the Weil pairing on $J[C]$. For any \verde{points} $R,S\in J[\ell]$, the Weil pairing satisfies $e_{\ell}(R,S)^q=e_{\ell}(\pi(R),\pi(S))$. Considering that  $\pi(D_1) = A D_1, \pi(D_i) = A D_i + D_{i-1}$ for $i \in \{2, 3, 4, 5\}$ and $\pi(D_6) = A D_6$, then $e_{\ell}(D_1,D_i)=1$ for all $i \in \{2,3,4,5\}$. 
In fact, we have that
\[
e_\ell(D_1,D_i)^q
= e_\ell(AD_1,AD_i+D_{i-1})
= e_\ell(D_1,D_i)^{A^2} e_\ell(D_1,D_{i-1})^A.
\]
We know that $q \equiv A^2 \pmod{\ell}$ and $e_{\ell}(\cdot,\cdot)$ takes values in the $\ell$-th roots of unity in $\overline{\mathbb{F}}_{q}$, so $$1 = e_{\ell}(D_1, D_{i-1})^A.$$ Since $A\in \mathbb{F}_{\ell}^{\times}$, this implies $e_\ell(D_1,D_{i-1})=1.$ Applying the same argument to $e_{\ell}(D_2,D_4)$ and $e_{\ell}(D_3,D_4)$, we successively obtain $e_\ell(D_2,D_3)=e_\ell(D_2,D_4)=1.$ Finally, substituting these relations into the identity for $e_{\ell}(D_2,D_5)$, we deduce $e_\ell(D_1,D_5)=1.$\\
Finally, from
\begin{eqnarray*}
    e_{\ell}(D_2,D_6)^q&=&e_{\ell}(AD_2+D_1,AD_6)\\
    &=&e_{\ell}(D_2,D_6)^{A^2}e_{\ell}(D_1,D_6)^{A}
\end{eqnarray*}
we obtain $e_\ell(D_1,D_6)=1$,
which contradicts the non-degeneracy of the Weil pairing.

For the remaining patterns, the same principle applies. In fact, in both cases there is a Jordan block of dimension 3. Let $\{D_1,D_2,D_3\}$ be a basis of the invariant subspace such that the action of $\pi$ is the following: $\pi(D_1)=AD_1, \pi(D_2)=AD_2+D_1$ and $\pi(D_3)=AD_3+D_2$. As in the previous cases, it is possible to prove that $e_{\ell}(D_1,D_2)=e_{\ell}(D_1,D_3)=1$. If the Jordan form is $[1]^3[1][1][1]$, then there is a basis $\{D_1,D_2,D_3,F_4,F_5,F_6\}$
such that $\pi(F_i)=AF_i$, for $i=4,5,6$. It follows that $e_{\ell}(D_1,F_i)=1$. 
Finally, if we are in the case $[1]^3[1]^2[1]$, the block of dimension 2 has a basis $\{E_1,E_2\}$ such that $\pi(E_1)=AE_1,\pi(E_2)=AE_2+E_1$ and  $e_{\ell}(D_1,E_1)=e_{\ell}(D_1,E_2)=1$.

\end{proof}

\begin{corollary} The different factorization types  of the characteristic polynomial are listed in Table \ref{tablatipos}.
     \begin{table}[]
         \centering
\begin{tabular}{|c|l|}
\hline
\textbf{Type} & \textbf{Factorization of the characteristic polynomial} \\ \hline
I   & $[6]$ \\ \hline
II  & $[3][3]$ \\ \hline
III & $[4][2]$ \\ \hline
IV  & $[4][1][1],\; [4][1]^2$ \\ \hline
V   & $[2][2][2],\; [2]^2[2],\; [2]^3$ \\ \hline
VI  & $[2][2][1][1],\; [2][2][1]^2,\; [2]^2[1][1],\; [2]^2[1]^2$ \\ \hline
VII & $[2][1]^4,\; [2][1]^2[1]^2,\; [2][1]^2[1][1],\; [2][1][1][1][1]$ \\ \hline
VIII& $[1]^2[1]^4,\; [1]^2[1]^2[1]^2,\; [1]^2[1]^2[1][1],\; [1]^2[1][1][1][1],\; [1]^3[1]^3,$\\ & $[1][1][1][1][1][1]$,  $[1]^6$ \\ \hline
\end{tabular}
\caption{Types of factorization of characteristic polynomial \azul{of the Frobenius endomorphism for genus 3 curves.}}
         \label{tablatipos}
     \end{table}
\end{corollary} 

\subsection{Factorization pattern}

Having classified the possible matrices, we can now show how to obtain factorization patterns from them.
 Following the approach of Gaudry and Schost \cite{GS05}, we provide the details for the case with  $P(x)=p(x)^3$ where $p(x)$ is  an irreducible polynomial of degree 2. The remaining cases follow by analogy. 
\\ 
We will use the following lemma
\begin{lemma}{ \cite[Lemma 1]{riquelme2016}.} \label{lemma:orbit} Let $D$ be a point in $J[\ell]$ and let $$V_D:=Span_{\F_\ell}\{\pi^n(D), n \in \mathbb{N} \}.$$ Let $P$ be the minimal polynomial of $\pi$ restricted to $V_D$. Then the degree of extension of $\F_q$ where $D$ is defined is $$o(P):=min\{k \in \mathbb{N}^*:x^k-1=0 \ \pmod{P} \}.$$\end{lemma}

\verde{ We first provide a counting lemma valid for a general companion block, from which the individual cases can be deduced by combining such blocks.}

\begin{lemma}\label{lemma:orbitlayers}
\azul{Let $A_d$ be the block associated with an irreducible polynomial
$p(x)$ of degree $d$, and let $n$ be its length.
For each $1\le k\le n$, the number of elements annihilated by
$p(x)^k$ but not by $p(x)^{k-1}$ is
\[
\ell^{d(k-1)}(\ell^d-1).
\]}
\end{lemma}
\begin{proof}
The subspace annihilated by $p(x)^k$ has dimension $kd$ and hence
contains $\ell^{kd}-1$ nonzero elements.
Similarly, the subspace annihilated by $p(x)^{k-1}$ has dimension
$(k-1)d$ and contains $\ell^{(k-1)d}-1$ nonzero elements.
Therefore the number of elements annihilated by $p(x)^k$ but not by
$p(x)^{k-1}$ is
\[
(\ell^{kd}-1)-(\ell^{(k-1)d}-1)
=
\ell^{(k-1)d}(\ell^d-1).
\]
\end{proof}

\begin{defi}
    The companion matrix of a monic polynomial $p(x)=a_0+a_1x+\ldots+a_{n-1}x^{n-1}+x^n$ is given by
$$C[p]=\begin{bmatrix}
0 & 1 & 0 & \cdots & 0 & 0 \\
0 & 0 & 1 & \cdots & 0 & 0 \\
\vdots & \vdots & \ddots & \ddots & \vdots & \vdots \\
0 & 0 & \cdots & 0 & 1 & 0 \\
0 & 0 & \cdots & 0 & 0 & 1 \\
-a_0 & -a_1 & \cdots & -a_{n-3} & -a_{n-2} & -a_{n-1}
\end{bmatrix}.$$
\end{defi}

Having established the degree of the field of definition for a divisor (Lemma \ref{lemma:orbit}) and the count of elements at a specific level within a single companion block (Lemma \ref{lemma:orbitlayers}), we now extend this analysis to the entire vector space. Since the Frobenius matrix $M$ decomposes as a direct sum of companion blocks, the choices of components in each block are independent. 
The following proposition provides the systematic method to compute the exact number of $\ell$-torsion elements defined over a specific extension by combining these independent blocks.

\begin{proposition}\label{prop:galois_orbits}
 Let $p_1(x),\dots,p_n(x)$ be distinct irreducible polynomials of degrees 
$d_1,\dots,d_n$, respectively, and consider the  matrix $$M=\operatorname{diag}\!\left(
C[p_1^{e_{1,1}}],\dots,C[p_1^{e_{1,m_1}}],
\dots,
C[p_n^{e_{n,1}}],\dots,C[p_n^{e_{n,m_n}}]
\right).$$
For elements with nonzero components in blocks 
$C[p_{i_1}^{e_{i_1,j_1}}],\dots,C[p_{i_r}^{e_{i_r,j_r}}]$ 
at levels $k_1,\dots,k_r$, (i.e., with minimal 
polynomials exactly $p_{i_s}^{k_s}(x)$ respectively), the following statements hold:
\begin{itemize}
    \item [(i)] Their minimal polynomial is $\operatorname{lcm}\!\left(p_{i_1}^{k_1}(x),\dots,p_{i_r}^{k_r}(x)\right),$

\item [(ii)] They are defined over an extension of degree
$o\!\left(\operatorname{lcm}\!\left(p_{i_1}^{k_1}(x),
\dots,p_{i_r}^{k_r}(x)\right)\right)$. 

\item[(iii)] The total number of such elements is $\prod_{s=1}^{r}\ell^{d_{i_s}(k_s-1)}(\ell^{d_{i_s}}-1)$ .
\end{itemize}

\end{proposition}

\begin{proof}

 For each block $C[p_i^{e_{i,j}}]$, Lemma~\ref{lemma:orbitlayers} gives 
$\ell^{d_i(k-1)}(\ell^{d_i}-1)$ elements at level $k$. Since the matrix $M$ decomposes as a direct sum, a vector with nonzero component at level $k_s$ in block $C[p_{i_s}^{e_{i_s,j_s}}]$ for $s=1,\dots,r$ 
has minimal polynomial
\[
\operatorname{lcm}\left(p_{i_1}^{k_1}(x),\dots,p_{i_r}^{k_r}(x)\right).
\]
The choices in each block are independent, so the number of such elements is
\[
\prod_{s=1}^{r}\ell^{d_{i_s}(k_s-1)}(\ell^{d_{i_s}}-1),
\]
and the corresponding orbit degree is 
$o\!\left(\operatorname{lcm}\left(p_{i_1}^{k_1}(x),\dots,
p_{i_r}^{k_r}(x)\right)\right)$ by Lemma~\ref{lemma:orbit}.\end{proof}

\verde{The fact that the minimal polynomial of a vector in a direct sum of invariant subspaces is the least common multiple of the minimal polynomials of its components follows from the order properties of modules over a principal ideal domain (for more details, see \cite[Theorems 6.2 and  7.14]{roman2008}).}

The matrix representing the Frobenius endomorphism with characteristic polynomial $P(x)$ admits three possible Jordan forms:
$$(i) \begin{pmatrix}
A_2 & 0 & 0 \\
0   & A_2 & 0 \\
0   & 0   & A_2
\end{pmatrix}, \quad (ii)
\begin{pmatrix}
A_2 & I_2 & 0 \\
0   & A_2 & I_2 \\
0   & 0   & A_2
\end{pmatrix},\quad (iii)
\begin{pmatrix}
A_2 & 0 & 0 \\
0   & A_2 & I_2 \\
0   & 0   & A_2
\end{pmatrix}.$$
We now describe the factorization of the Galois orbits for case $(iii)$, as the remaining cases are similar.\\ 

We have that $M = \operatorname{diag}(C[p^2], C[p])$.
\begin{itemize}
    \item In the first level ($k=1$) of the block $C[p^2]$, the elements have minimal polynomial $p(x)$. By Lemma~\ref{lemma:orbitlayers}, there are $ \ell^{2(1-1)}(\ell^2-1)=\ell^2-1$
    such elements, all defined over an extension of degree $o(p(x))$.
    \item In the second level ($k=2$) of the block $C[p(x)^2]$, the elements have minimal polynomial $p(x)^2$. By Lemma~\ref{lemma:orbitlayers}, there are $    \ell^{2(2-1)}(\ell^2-1)=\ell^2(\ell^2-1)$
    such elements, all defined over an extension of degree $o(p(x)^2)$.
    \item In the block $C[p]$, the elements have minimal polynomial $p(x)$. By Lemma~\ref{lemma:orbitlayers}, there are
$\ell^{2(1-1)}(\ell^2-1)=\ell^2-1$
    such elements, all defined over an extension of degree $o(p(x))$.
    \item Consider elements with nonzero components in both the first level of $C[p^2]$ and the block $C[p]$. Since $    \operatorname{lcm}(p(x),p(x))=p(x)$,
    their minimal polynomial is $p(x)$. Hence there are
     $(\ell^2-1)^2$
    such elements, all defined over an extension of degree $o(p(x))$.
    \item Consider elements with nonzero components in both the second level of $C[p^2]$ and the block $C[p]$. Since $    \operatorname{lcm}(p(x)^2,p(x))=p(x)^2,$
    their minimal polynomial is $p(x)^2$. Hence there are $    \ell^2(\ell^2-1)^2$
    such elements, all defined over an extension of degree $o(p(x)^2)$.
\end{itemize}
The corresponding elements are grouped according to orbit size:
$(\ell^2-1)+(\ell^2-1)+(\ell^2-1)^2
=\ell^4-1$ of degree $o(p(x))$ and
$\ell^2(\ell^2-1)+\ell^2(\ell^2-1)^2
=
\ell^4(\ell^2-1)$ of degree $o(p(x)^2)$.  \\
Therefore the $\ell$-torsion Galois orbits  decomposition is  
\[
\bigl(o(p(x))\bigr)^{\frac{\ell^4-1}{o(p(x))}}
\bigl(o(p(x)^2)\bigr)^{\frac{\ell^4(\ell^2-1)}{o(p(x)^2)}}.
\]
\begin{example}\label{examplew}
Consider the field $\mathbb{F}_5$ and the curve 
\begin{align*}
\C : y^2&= x^7 + 4x^3 + 3x;
\end{align*} 
over $\mathbb{F}_5.$. \verde {The characteristic polynomial of Frobenius acting on $T_{\ell}(\pi)$} is $(x^2+2x+2)^3$. \\
Over $\mathbb{F}_3$, the polynomial
$p(x)=x^2+2x+2$ is irreducible.
The elements contained in the subspace corresponding to the first blocks in rational form are defined over an extension of degree $o(p(x)) = 8$, while the remaining elements, which are in the next layer, are defined over an extension of degree $o(p(x)^2) = 24$.
Then, its factorization of the 3-torsion Galois orbits
is of the form $(8)^{10}(24)^{27}$. 
There are $10=\dfrac{3^4-1}{8}$ Galois orbits of length $8$ and $27=\dfrac{3^4(3^2-1)}{24}$ Galois orbits of length
$24$. Therefore, the number of nonzero $3$-torsion points accounted for is
\[
10\cdot 8 + 27\cdot 24 = 80+648=728=3^6-1.
\]
This agrees with the fact that $J[3]$ is a $6$-dimensional
$\mathbb F_3$-vector space.
Consequently, the $3$-torsion has rank $4$ over $\mathbb F_{p^8}$ and full
rank $6$ over $\mathbb F_{p^{24}}$. Hence $J[3](\mathbb F_{p^8})\simeq (\mathbb Z/3\mathbb Z)^4$ and 
$J[3](\mathbb F_{p^{24}})\simeq (\mathbb Z/3\mathbb Z)^6.$
\end{example}

\section{Upper bounds for extension degrees in genus 3}

Following Matsuo’s approach, we analyze the upper bound for the extension degree of the field of definition of \(J[\ell]\) for each type.
Throughout this section we assume that $\ell\neq 2$, unless the case $\ell=2$ is treated separately. For each factorization type of the characteristic polynomial modulo $\ell$, we proceed as follows.
First, we determine a finite extension of $\mathbb{F}_\ell$ over which the matrix representing the Frobenius endomorphism becomes upper triangular, and identify the semisimple part of its action on  J[$\ell$].
We then bound the exponent needed to annihilate the nilpotent contribution, which depends on the maximal size of the Jordan blocks.
Combining these two ingredients yields an explicit upper bound for the degree of the field of definition of J[$\ell$].
 We denote by
\[
e = \#\langle \tilde{q} \rangle,\quad
e_n = \#\langle -\tilde{q} \rangle,
\]
the multiplicative orders of $\tilde q$ and $-\tilde q$ in
$\mathbb F_\ell^\times$.
If $\tilde q$ is a quadratic residue modulo $\ell$, we denote by $e_r=\#\langle\sqrt{\tilde q}\rangle$
the multiplicative order of a fixed square root
$\sqrt{\tilde q}\in\mathbb F_\ell^\times$.
By Fermat's little theorem, all these orders divide $\ell - 1$.
In order to obtain uniform bounds, we  assume
$e = e_n = e_r = \ell - 1$.

\begin{center}
\begin{table}
\begin{tabular}{llllll} 
\toprule
Type& Char. poly & Max. ext. deg  & Worst case &  &  \\ \midrule
\addlinespace[0.5em]
Type I&$[6]$ & $e(\ell^3+1)$ & $\ell^4 - \ell^3 + \ell - 1$ &  \\
\addlinespace[0.5em]
Type II&$[3][3]$ & $(\ell-1)(\ell^2+\ell+1) $ & $\ell^3-1$ &  \\
\addlinespace[0.5em]
Type III&$[4][2]$ &  $ \lcm(e(\ell^2+1),e(\ell+1)$)& $(\ell^4-1)/2$ \\
\addlinespace[0.5em]
Type IV&$[4][1][1]$ & $e\ell(\ell^2+1)$& $\ell^4-\ell^3+\ell^2-\ell$\\
Type V&$[2][2][2]$ & $e \ell(\ell+1)$ &  $\ell^3-\ell$  &  \\
Type VI&$[2][2][1][1]$ & $(\ell-1) \ell(\ell+1)$  & $\ell^3-\ell$  &  \\
Type VII&$[2][1][1][1][1]$ & $\lcm(e_r\ell^2,e(\ell+1))$  & $(\ell-1)\ell^2(\ell+1)/2$  &  \\
Type VIII&$[1][1][1][1][1][1]$ & $e_r\ell^2$ & $(\ell-1)\ell^2$  &  \\
\bottomrule
\end{tabular} 
\caption{ Upper bounds of extension degrees.}\label{tabla2}
\end{table}
\end{center}

\subsection*{Type I}
In this type, \(\tilde{P}\) has six roots
\(\tilde{\alpha}, \tilde{\alpha}^\ell, \tilde{\alpha}^{\ell^2}, \tilde{\alpha}^{\ell^3}, \tilde{\alpha}^{\ell^4}, \tilde{\alpha}^{\ell^5}, \in \mathbb{F}_{\ell^6} \setminus \mathbb{F}_{\ell^2}\) and \(\mathbb{F}_{\ell^6} \setminus \mathbb{F}_{\ell^3}\),
which are all distinct.  
Therefore, \(A\) can be defined over \(\mathbb{F}_{\ell^6}\) as

\[
A =
\begin{pmatrix}
\tilde{\alpha} & 0 & 0 & 0 & 0 & 0\\
0 & \tilde{\alpha}^\ell & 0 & 0 & 0 & 0\\
0 & 0 & \tilde{\alpha}^{\ell^2} & 0 & 0 & 0\\
0 & 0 & 0 & \tilde{\alpha}^{\ell^3} & 0 & 0\\
0 & 0 & 0 & 0 & \tilde{\alpha}^{\ell^4} & 0\\
0 & 0 & 0 & 0 & 0 & \tilde{\alpha}^{\ell^5} \\
\end{pmatrix}.
\]
From Theorem \ref{roots}, we can see that

\[
\tilde{\alpha}^{\ell^3+1} = \tilde{q} \in \mathbb{F}_\ell,
\]
so that

\[
A^{\ell^3+1} =
\begin{pmatrix}
\tilde{q} & 0 & 0 & 0 & 0 & 0\\
0 & \tilde{q} & 0 & 0 & 0 & 0\\
0 & 0 & \tilde{q} & 0 & 0 & 0\\
0 & 0 & 0 & \tilde{q} & 0 & 0\\
0 & 0 & 0 & 0 & \tilde{q} & 0\\
0 & 0 & 0 & 0 & 0 & \tilde{q} \\
\end{pmatrix}.
\]
For \(D \in J[\ell]\), we have

\[
D^{q^{e(\ell^3+1)}} - D = 0,
\]
because
\[
T^{e(\ell^3+1)} = P A^{e(\ell^3+1)} P^{-1} = I_6.
\]
Therefore, we obtain
\[
J[\ell] \subset J(\mathbb{F}_{q^{e(\ell^3+1)}}) \subset J(\mathbb{F}_{q^{\ell^4 - \ell^3 + \ell - 1}}).
\]

\subsection*{Type II}
In this type, \(\tilde{P}\) has 4 roots
\(\tilde{\alpha}_1, \tilde{\alpha}_1^\ell, \tilde{\alpha}_1^{\ell^2}, \tilde{\alpha}_2, \tilde{\alpha}_2^\ell\in \mathbb{F}_{\ell^2},\tilde{\alpha}_1^{\ell^3}\in \mathbb{F}_{\ell^3} \setminus \mathbb{F}_\ell\),
which are all distinct.  
Therefore, \(A\) can be defined over \(\mathbb{F}_{\ell^2}\) as

\[
A =
\begin{pmatrix}
\tilde{\alpha}_1 & 0 & 0 & 0 & 0 & 0\\
0 & \tilde{\alpha}_1^\ell & 0 & 0 & 0 & 0\\
0 & 0 & \tilde{\alpha}_1^{\ell^2} & 0 & 0 & 0\\
0 & 0 & 0 & \tilde{\alpha}_2 & 0 & 0\\
0 & 0 & 0 & 0 & \tilde{\alpha}_2^\ell & 0\\
0 & 0 & 0 & 0 & 0 & \tilde{\alpha}_2^{\ell^2}
\end{pmatrix}.
\]
Moreover,
\[
N_{\mathbb{F}_{l^3}/\mathbb{F}_\ell}(\tilde{\alpha}_i) = \tilde{\alpha}_i^{\ell^2 + \ell + 1} \in \mathbb{F}_\ell
\quad (i = 1, 2).
\]
It follows that \(A^{(\ell - 1)(\ell^2 + \ell + 1)} = I_6\), and therefore
\[
J[\ell] \subset J(\mathbb{F}_{q^{\ell^3 - 1}}).
\]

\subsection*{Type III}
In this type, \(\tilde{P}\) has 4 roots 
\(\tilde{\alpha}_1, \tilde{\alpha}_1^\ell, \tilde{\alpha}_1^{\ell^2}, \in \mathbb{F}_{\ell^4} \setminus \mathbb{F}_{\ell^2}\), and $\tilde{\alpha}_2, \tilde{\alpha}_2^\ell \in \mathbb{F}_{\ell^2} \setminus \mathbb{F}_{\ell}$ 
which are all distinct.  
Therefore, \(A\) can be defined over \(\mathbb{F}_{\ell^4}\) as

\[
A =
\begin{pmatrix}
\tilde{\alpha}_1 & 0 & 0 & 0 & 0 & 0\\
0 & \tilde{\alpha}_1^\ell & 0 & 0 & 0 & 0\\
0 & 0 & \tilde{\alpha}_1^{\ell^2} & 0 & 0 & 0\\
0 & 0 & 0 & \tilde{\alpha}_1^{\ell^3} & 0 & 0\\
0 & 0 & 0 & 0 & \tilde{\alpha}_2 & 0\\
0 & 0 & 0 & 0 & 0 & \tilde{\alpha}_2^{\ell}
\end{pmatrix}.
\]
 Let $A_4$ (respectively. $A_2$) denote the companion matrix of the irreducible factor
of degree $4$ (respectively $2$) of \(\tilde P\).
It follows that \(A_4^{(\ell - 1)(\ell^2 + 1)} = I_4\) and \(A_2^{(\ell - 1)(\ell + 1)} = I_2\),  therefore taking the least common multiple of these two orders yields

\[
J[\ell] \subset J(\mathbb{F}_{q^{\lcm(\ell^2+1,\ell+1)}}) \subset J(\mathbb{F}_{q^{(\ell^4-1)/2}}) .
\]

\subsection*{Type IV}
In this type, \(\tilde{P}\) has 4 roots 
\(\tilde{\alpha}_1, \tilde{\alpha}_1^\ell, \tilde{\alpha}_1^{\ell^2},\tilde{\alpha}_1^{\ell^3} \in \mathbb{F}_{\ell^4} \setminus \mathbb{F}_{\ell^2}\), and 2 roots $\tilde{\alpha}_2, \tilde{\alpha}_2^\ell \in \mathbb{F}_{\ell^2} \setminus \mathbb{F}_{\ell}$ 
which are all distinct.  
Therefore, \(A\) can be defined over \(\mathbb{F}_{\ell^4}\) as

\[
(i)\quad A =
\begin{pmatrix}
\tilde{\alpha}_1 & 0 & 0 & 0 & 0 & 0\\
0 & \tilde{\alpha}_1^\ell & 0 & 0 & 0 & 0\\
0 & 0 & \tilde{\alpha}_1^{\ell^2} & 0 & 0 & 0\\
0 & 0 & 0 & \tilde{\alpha}_1^{\ell^3} & 0 & 0\\
0 & 0 & 0 & 0 & \tilde{\alpha}_2 & 0\\
0 & 0 & 0 & 0 & 0 & \tilde{\alpha}_2
\end{pmatrix},\quad (ii)\quad A =
\begin{pmatrix}
\tilde{\alpha}_1 & 0 & 0 & 0 & 0 & 0\\
0 & \tilde{\alpha}_1^\ell & 0 & 0 & 0 & 0\\
0 & 0 & \tilde{\alpha}_1^{\ell^2} & 0 & 0 & 0\\
0 & 0 & 0 & \tilde{\alpha}_1^{\ell^3} & 0 & 0\\
0 & 0 & 0 & 0 & \tilde{\alpha}_2 & 1\\
0 & 0 & 0 & 0 & 0 & \tilde{\alpha}_2
\end{pmatrix},
\]
\[
(iii) \quad A =
\begin{pmatrix}
\tilde{\alpha}_1 & 0 & 0 & 0 & 0 & 0\\
0 & \tilde{\alpha}_1^\ell & 0 & 0 & 0 & 0\\
0 & 0 & \tilde{\alpha}_1^{\ell^2} & 0 & 0 & 0\\
0 & 0 & 0 & \tilde{\alpha}_1^{\ell^3} & 0 & 0\\
0 & 0 & 0 & 0 & \tilde{\alpha}_2 & 0\\
0 & 0 & 0 & 0 & 0 & \tilde{\alpha}_3
\end{pmatrix},
\]
As in Type III, taking the least common multiple of these two orders yields
\[
J[\ell] \subset J(\mathbb{F}_{q^{\lcm(\ell^2+1,\ell+1)}}) \subset J(\mathbb{F}_{q^{(\ell^4-1)/2}}) .
\]

\subsection*{Type V}
In this type, $\tilde{P}$  has $6$ roots in
$\mathbb{F}_{\ell^2} \setminus \mathbb{F}_\ell$.
The cases  $[2]^3$,  $[2]^2[2]$  and  $[2][2][2]$  are possible.
Because the cases  $[2]^3$ and  $[2]^2[2]$  have roots with multiplicity greater than~$1$,
the corresponding matrices may admit a non-trivial nilpotent part,
which can increase the extension degree.
We therefore analyze the worst cases.
Here $\tilde{\alpha}$  and  $\tilde{\alpha}_1$ denote elements of
$\mathbb{F}_{\ell^2}\setminus\mathbb{F}_\ell$,
and $\tilde{\alpha}^\ell$ and $\tilde{\alpha}_1^\ell$ their Galois conjugates. Among all configurations compatible with this factorization pattern,
cases (i) and (ii) yield the largest possible nilpotent blocks and therefore
lead to the largest extension degree.

\[
(i)\quad A =
\begin{pmatrix}
\tilde{\alpha} & 1 & 0 & 0 & 0 & 0\\
0 & \tilde{\alpha} & 0 & 0 & 0 & 0\\
0 & 0 & \tilde{\alpha}^\ell & 1 & 0 & 0\\
0 & 0 & 0 & \tilde{\alpha}^{\ell} & 0 & 0\\
0 & 0 & 0 & 0 & \tilde{\alpha}_1 & 0\\
0 & 0 & 0 & 0 & 0 & \tilde{\alpha}_1^{\ell}
\end{pmatrix},\quad (ii)\quad A =
\begin{pmatrix}
\tilde{\alpha} & 1 & 0 & 0 & 0 & 0\\
0 & \tilde{\alpha} & 1 & 0 & 0 & 0\\
0 & 0 & \tilde{\alpha} & 0 & 0 & 0\\
0 & 0 & 0 & \tilde{\alpha}^{\ell} & 1 & 0\\
0 & 0 & 0 & 0 & \tilde{\alpha}^{\ell} & 1\\
0 & 0 & 0 & 0 & 0 & \tilde{\alpha}^{\ell}
\end{pmatrix},
\]

\begin{itemize}
    \item[(i)] In case (i) \[
A^k =
\begin{pmatrix}
\tilde{\alpha}^{k} & k\tilde{\alpha}^{k-1} & 0 & 0 & 0 & 0\\
0 & \tilde{\alpha}^{k} & 0 & 0 & 0 & 0\\
0 & 0 & \tilde{\alpha}^{\ell k} & k\tilde{\alpha}^{\ell(k-1)} & 0 & 0\\
0 & 0 & 0 & \tilde{\alpha}^{\ell k} & 0 & 0\\
0 & 0 & 0 & 0 & \tilde{\alpha}_1^{k} & 0\\
0 & 0 & 0 & 0 & 0 & \tilde{\alpha}_1^{\ell k}
\end{pmatrix}.
\]
$A^{e\ell(\ell+1)} = I_{6}$
 or 
$A^{e_n\ell(\ell+1)} = I_{6}  $. 
 Therefore we have
\[
J[\ell] \subset J(\mathbb{F}_{q^{e\ell(\ell+1)}}) \subset J(\mathbb{F}_{q^{\ell^3-\ell}}).
\]
or \[
J[\ell] \subset J(\mathbb{F}_{q^{e_n\ell(\ell+1)}}) \subset J(\mathbb{F}_{q^{\ell^3-\ell}}).
\]
 \item[(ii)] In case (ii) 
 \[
A^k =
\begin{pmatrix}
\tilde{\alpha}^{k} &
\binom{k}{1}\tilde{\alpha}^{k-1} &
\binom{k}{2}\tilde{\alpha}^{k-2} &
0 & 0 & 0 \\

0 &
\tilde{\alpha}^{k} &
\binom{k}{1}\tilde{\alpha}^{k-1} &
0 & 0 & 0 \\

0 & 0 &
\tilde{\alpha}^{k} &
0 & 0 & 0 \\[0.4em]

0 & 0 & 0 &
\tilde{\alpha}^{lk} &
\binom{k}{1}\tilde{\alpha}^{l(k-1)} &
\binom{k}{2}\tilde{\alpha}^{l(k-2)} \\

0 & 0 & 0 & 0 &
\tilde{\alpha}^{lk} &
\binom{k}{1}\tilde{\alpha}^{l(k-1)} \\

0 & 0 & 0 & 0 & 0 &
\tilde{\alpha}^{lk}
\end{pmatrix}.
\]
Therefore,
\[
J[\ell] \subset J(\mathbb{F}_{q^{e\ell(\ell+1)}}) \subset J(\mathbb{F}_{q^{\ell^3-\ell}}) .
\]
\begin{remark}
If $\ell= 2$, then
    \[
J[\ell] \subset J(\mathbb{F}_{q^{e\ell^2(\ell+1)}}) \subset J(\mathbb{F}_{q^{\ell^4-\ell^2}}) .
\]
\end{remark} 
\end{itemize}

\subsection*{Type VI}
In this type, $\tilde{P}$  has 4 roots in $ \mathbb{F}_{\ell^2} \setminus \mathbb{F}_\ell$ and 2 roots in $\mathbb{F}_\ell$.   
The cases  $[2]^2[1]^2, [2]^2[1][1]$, $ [2][2][1]^2,[2][2][1][1] $  are possible.
Because the cases $[2]^2[1]^2,[2]^2[1][1], [2][2][1]^2$  have roots with multiplicity greater than 1. The possible matrices could have a nilpotent part that increases the bound, so we will analyze the worst cases.
 $\tilde{\alpha}, \tilde{\alpha}^\ell, \tilde{\alpha}, \tilde{\alpha}^\ell, \tilde{\alpha}_1, \tilde{\alpha}_1, \in  \mathbb{F}_\ell$ 

\[
 A =
\begin{pmatrix}
\tilde{\alpha} & 1 & 0 & 0 & 0 & 0\\
0 & \tilde{\alpha} & 0 & 0 & 0 & 0\\
0 & 0 & \tilde{\alpha}^\ell & 1 & 0 & 0\\
0 & 0 & 0 & \tilde{\alpha}^{\ell} & 0 & 0\\
0 & 0 & 0 & 0 & \tilde{\alpha}_1 & 1\\
0 & 0 & 0 & 0 & 0 & \tilde{\alpha}_1
\end{pmatrix}
\]

\[
A^k =
\begin{pmatrix}
\tilde{\alpha}^{k} & k\tilde{\alpha}^{k-1} & 0 & 0 & 0 & 0\\
0 & \tilde{\alpha}^{k} & 0 & 0 & 0 & 0\\
0 & 0 & \tilde{\alpha}^{\ell k} & k\tilde{\alpha}^{l(k-1)} & 0 & 0\\
0 & 0 & 0 & \tilde{\alpha}^{\ell k} & 0 & 0\\
0 & 0 & 0 & 0 & \tilde{\alpha}_1^{k} & k\tilde{\alpha}_1^{k-1}\\
0 & 0 & 0 & 0 & 0 & \tilde{\alpha}_1^{k}
\end{pmatrix}.
\]
therefore
\[
J[\ell] \subset J(\mathbb{F}_{q^{e\ell(\ell+1)}}) \subset J(\mathbb{F}_{q^{\ell^3-\ell}}) .
\]
\subsection*{Type VII}
In this type, $\tilde{P}$ has 2 roots $\tilde{\alpha_1},\tilde{\alpha_1}^{\ell}$
$\in \mathbb{F}_{\ell^2}\setminus \mathbb{F}_{\ell}$   and a
quadruple root $\tilde{\alpha}\in \mathbb{F}_{\ell} $. Therefore A can be given over $\mathbb{F}_{\ell^2} $ in the worst case as either
\[
A =
\begin{pmatrix}
\tilde{\alpha} & 1 & 0 & 0 & 0 & 0\\
0 & \tilde{\alpha} & 1 & 0 & 0 & 0\\
0 & 0 & \tilde{\alpha} & 1 & 0 & 0\\
0 & 0 & 0 & \tilde{\alpha} & 0 & 0\\
0 & 0 & 0 & 0 & \tilde{\alpha}_1 & 0\\
0 & 0 & 0 & 0 & 0 & \tilde{\alpha}_1^{l}
\end{pmatrix}.
\]
\[
A^k =
\begin{pmatrix}
\tilde{\alpha}^{k} &
\binom{k}{1}\tilde{\alpha}^{k-1} &
\binom{k}{2}\tilde{\alpha}^{k-2} &
\binom{k}{3}\tilde{\alpha}^{k-3} &
0 & 0
\\
0 &
\tilde{\alpha}^{k} &
\binom{k}{1}\tilde{\alpha}^{k-1} &
\binom{k}{2}\tilde{\alpha}^{k-2} &
0 & 0
\\
0 & 0 &
\tilde{\alpha}^{k} &
\binom{k}{1}\tilde{\alpha}^{k-1} &
0 & 0
\\
0 & 0 & 0 &
\tilde{\alpha}^{k} &
0 & 0
\\
0 & 0 & 0 & 0 &
\tilde{\alpha}_1^{k} &
0
\\
0 & 0 & 0 & 0 & 0 &
\tilde{\alpha}_1^{lk}
\end{pmatrix}
\]
holds for any $k \in \mathbb{N}$. Therefore we have, for  $ \ell \not\in\{2, 3\}$,
\begin{equation*}
    J[\ell] \subset J (\mathbb{F}_{q^{\lcm(e_r\ell,e(\ell+1))}} ) ,
\end{equation*}
and for $\ell\in \{2,3\}$
\begin{equation*}
    J[\ell] \subset J (\mathbb{F}_{q^{\lcm(e_r\ell^2,e(\ell+1))}} ). 
\end{equation*}
\subsection*{Type VIII}
In this type, $\tilde{P}$ has a sixfold  root $\tilde{\alpha}=\sqrt{\tilde{q}}\in \mathbb{F}_\ell$. In the worst nilpotent case, $A$ can be given over $\mathbb{F}_\ell$ as either
\[
A =
\begin{pmatrix}
\tilde{\alpha} & 1 & 0 & 0 & 0 & 0\\
0 & \tilde{\alpha} & 1 & 0 & 0 & 0\\
0 & 0 & \tilde{\alpha} & 1 & 0 & 0\\
0 & 0 & 0 & \tilde{\alpha} & 1 & 0\\
0 & 0 & 0 & 0 & \tilde{\alpha} & 1\\
0 & 0 & 0 & 0 & 0 & \tilde{\alpha}
\end{pmatrix},
\]
and
\[
A^k =
\begin{pmatrix}
\tilde{\alpha}^{k} &
\binom{k}{1}\tilde{\alpha}^{k-1} &
\binom{k}{2}\tilde{\alpha}^{k-2} &
\binom{k}{3}\tilde{\alpha}^{k-3} &
\binom{k}{4}\tilde{\alpha}^{k-4} &
\binom{k}{5}\tilde{\alpha}^{k-5}
\\
0 &
\tilde{\alpha}^{k} &
\binom{k}{1}\tilde{\alpha}^{k-1} &
\binom{k}{2}\tilde{\alpha}^{k-2} &
\binom{k}{3}\tilde{\alpha}^{k-3} &
\binom{k}{4}\tilde{\alpha}^{k-4}
\\
0 & 0 &
\tilde{\alpha}^{k} &
\binom{k}{1}\tilde{\alpha}^{k-1} &
\binom{k}{2}\tilde{\alpha}^{k-2} &
\binom{k}{3}\tilde{\alpha}^{k-3}
\\[6pt]
0 & 0 & 0 &
\tilde{\alpha}^{k} &
\binom{k}{1}\tilde{\alpha}^{k-1} &
\binom{k}{2}\tilde{\alpha}^{k-2}
\\[6pt]
0 & 0 & 0 & 0 &
\tilde{\alpha}^{k} &
\binom{k}{1}\tilde{\alpha}^{k-1}
\\[6pt]
0 & 0 & 0 & 0 & 0 &
\tilde{\alpha}^{k}
\end{pmatrix}.
\]
holds for any $k \in \mathbb{N}$. Then, for  $ \ell \not\in\{2, 3,5\}$ we have, 
\begin{equation*}
    J[\ell] \subset J (\mathbb{F}_{q^{e_r\ell}} ) \subset J (\mathbb{F}_{q^{\ell^2-\ell}})
\end{equation*}
from $A^{e_r\ell} = I_4$, and for $\ell \in \{ 3,5\}$, 
\begin{equation*}
    J[\ell] \subset J (\mathbb{F}_{q^{e_r\ell^2}} ) \subset J (\mathbb{F}_{q^{\ell^3-\ell^2}}).
\end{equation*}
Finally, for $\ell=2$
\begin{equation*}
    J[\ell] \subset J (\mathbb{F}_{q^{e_r\ell^3}} ) \subset J (\mathbb{F}_{q^{\ell^4-\ell^3}}).
\end{equation*}
\begin{theorem}\label{teo1}
If $ \ell > 2$, there exists a positive integer \(d \leq \ell^4 - \ell^3 + \ell - 1\) such that
\[
J[\ell] \subset J(\mathbb{F}_{q^d}).
\] 
\end{theorem} 
\begin{proof}
The bounds obtained in Types I--VIII are summarized in Table \ref{tabla2}.
The maximal value is attained in Type~I and equals $e(\ell^3+1)$.
Since $e\mid(\ell-1)$,
\[
e(\ell^3+1)\le (\ell-1)(\ell^3+1)
=\ell^4-\ell^3+\ell-1.
\]
This proves the theorem.
\end{proof}
The following example shows that the bound of Theorem~\ref{teo1} is attained for $\ell=3$.

\begin{example}\label{example2}
Consider $p=5$ and the curve 
\begin{align*}
\C : y^2&= x^7 + x^2 + 3x + 1
\end{align*} 
over $\mathbb{F}_p$. \verde{ The characteristic polynomial of Frobenius acting on $T_{\ell}(\pi)$} is $x^6 + 2x^4 + 2x^3 + x^2 + 2$. Then, its factorization of the 3-torsion Galois orbits
is of the form $(56)^{13}$. 
\azul{Consequently, the full $3$-torsion is defined over
$\mathbb F_{p^{56}}$. Hence
\[
J[3](\mathbb F_{p^{56}})
\simeq (\mathbb Z/3\mathbb Z)^6.
\]
In particular, the bound $56=3^4-3^3+3-1$ is attained.
}
\end{example}

\begin{remark}
    The bounds obtained in genus $2$, and $3$ suggest a common underlying pattern.
Let $J$ be an abelian variety of dimension $g$ defined over a finite field, and let
$\ell \neq p$ be a prime. The Frobenius action on $J[\ell]$ is represented by a $2g \times 2g$ matrix, whose largest possible Jordan block has size at most $2g$.
We observed for hyperelliptic curves of genus $2$ and $3$ that, for sufficiently large  $\ell$, the worst case arises in two situations: when $p(x)$  is irreducible of degree $2g$, or when $g=g_1+\cdots+g_r$ (corresponding to the case where the Jacobian is isogenous to a product of Jacobians of lower-genus curves).
In these cases, the bounds are respectively $(\ell-1)(\ell^g+1)$ and $(\ell-1)\displaystyle\prod_{i=1}^r(\ell^{g_i}+1),$
both of order $O(\ell^{g+1})$.
On the other hand, the nilpotent part is annihilated by an $\ell$-power $\ell^r$. 
In particular, for $\ell > g$, a power of order $\ell^{2}>2g$ suffices to annihilate
the nilpotent contribution. As a consequence, the nilpotent part cannot dominate the
growth coming from the largest semisimple block.

This provides a heuristic explanation for the appearance of bounds of order
$O(\ell^{g+1})$ in genus $2$ and $3$. 

\end{remark}


\newpage
\section*{Appendix A}\label{appendixa}
The first column describes the factorization pattern of the characteristic polynomial of the Frobenius endomorphism. The rows of the table are then indexed by the corresponding possible decompositions of the Frobenius endomorphism on $J[\ell]$, which are given in the second column.

\begin{center}
\fontsize{6pt}{6pt}\selectfont
\renewcommand{\arraystretch}{1.4}
\begin{tabular}{|c|c|cccc|}
\hline
\textbf{Type} & \textbf{Char. pol.} & \textbf{Matrix} & && \\ \hline

I & [6] &
$\begin{pmatrix}
A_6
\end{pmatrix}$ & & &\\ \hline

II & [3][3] &
$\begin{pmatrix}
A_3 & 0 \\
0   & B_3
\end{pmatrix}$ & & &\\ \hline

III & [4][2] &
$\begin{pmatrix}
A_4 & 0 \\
0   & B_2
\end{pmatrix}$ & & &\\ \hline

IV & [4][1][1] & 
$\begin{pmatrix}
A_4 & 0 \\
0   & \begin{matrix} b & 0 \\ 0 & c \end{matrix}
\end{pmatrix}$ &&&\\ \hline
 &  $[4][1]^2$ &
$\begin{pmatrix}
A_4 & 0 \\
0   & \begin{matrix} b & 0 \\ 0 & b \end{matrix}
\end{pmatrix}$ &
$\begin{pmatrix}
A_4 & 0 \\
0   & \begin{matrix} b & 1 \\ 0 & b \end{matrix}
\end{pmatrix}$ &&
\\ \hline

V & [2][2][2]&
$\begin{pmatrix}
A_2 & 0 & 0 \\
0   & B_2 & 0 \\
0   & 0   & C_2
\end{pmatrix}$ &
 & &
\\ \hline
 &  $[2]^2[2]$ &
$\begin{pmatrix}
A_2 & 0 & 0 \\
0   & A_2 & 0 \\
0   & 0   & B_2
\end{pmatrix}$ &
$\begin{pmatrix}
A_2 & * & 0 \\
0   & A_2 & 0 \\
0   & 0   & B_2
\end{pmatrix}$ &&
 \\ \hline
&  $[2]^3$ &
$\begin{pmatrix}
A_2 & 0 & 0 \\
0   & A_2 & 0 \\
0   & 0   & A_2
\end{pmatrix}$ &
$\begin{pmatrix}
A_2 & * & * \\
0   & A_2 & * \\
0   & 0   & A_2
\end{pmatrix}$ &
$\begin{pmatrix}
A_2 & 0 & 0 \\
0   & A_2 & * \\
0   & 0   & A_2
\end{pmatrix}$ &\\ \hline

VI & [2][2][1][1] &
$\begin{pmatrix}
A_2 & 0 & 0 \\
0   & B_2 & 0 \\
0   & 0   & \begin{matrix} c & 0 \\ 0 & d \end{matrix}
\end{pmatrix}$ &&
 & \\ \hline
 &  $[2][2][1]^2$&
$\begin{pmatrix}
A_2 & 0 & 0 \\
0   & B_2 & 0 \\
0   & 0   & \begin{matrix} c & 0 \\ 0 & c \end{matrix}
\end{pmatrix}$ &
$\begin{pmatrix}
A_2 & 0 & 0 \\
0   & B_2 & 0 \\
0   & 0   & \begin{matrix} c & 1 \\ 0 & c \end{matrix}
\end{pmatrix}$ & &\\ \hline
 &  $[2]^2[1][1]$&
$\begin{pmatrix}
A_2 & 0 & 0 \\
0   & A_2 & 0 \\
0   & 0   & \begin{matrix} b & 0 \\ 0 & c \end{matrix}
\end{pmatrix}$ &
$\begin{pmatrix}
A_2 & * & 0 \\
0   & A_2 & 0 \\
0   & 0   & \begin{matrix} b & 0 \\ 0 & b \end{matrix}
\end{pmatrix}$ && \\ \hline
 &  $[2]^2[1]^2$&
$\begin{pmatrix}
A_2 & 0 & 0 \\
0   & A_2 & 0 \\
0   & 0   & \begin{matrix} b & 0 \\ 0 & b\end{matrix}
\end{pmatrix}$ &
$\begin{pmatrix}
A_2 & 0 & 0 \\
0   & A_2 & 0 \\
0   & 0   & \begin{matrix} b & 1 \\ 0 & b \end{matrix}
\end{pmatrix}$ &$\begin{pmatrix}
A_2 & * & 0 \\
0   & A_2 & 0 \\
0   & 0   & \begin{matrix} b & 0 \\ 0 & b \end{matrix}
\end{pmatrix}$ &$\begin{pmatrix}
A_2 & * & 0 \\
0   & A_2 & 0 \\
0   & 0   & \begin{matrix} b & 1 \\ 0 & b \end{matrix}
\end{pmatrix}$ \\ \hline
\end{tabular}
\end{center}

\newpage

\begin{center}
\fontsize{6pt}{6pt}\selectfont
\renewcommand{\arraystretch}{1.4}
\begin{tabular}{|c|c|cccc|}
\hline
\textbf{Type} & \textbf{Char. pol.} & \textbf{Matrix} & && \\ \hline
 VII&  $[2][1]^4$ &
$\begin{pmatrix}
A_2 & 0 \\
0   & \begin{matrix} b & 0 &0&0\\ 0 & b&0&0 \\ 0 & 0&b&0 \\ 0 & 0&0&b \end{matrix}
\end{pmatrix}$ &$\begin{pmatrix}
A_2 & 0 \\
0   & \begin{matrix} b & 1 &0&0\\ 0 & b&0&0 \\ 0 & 0&b&0 \\ 0 & 0&0&b \end{matrix}
\end{pmatrix}$ &&\\&
 &$\begin{pmatrix}
A_2 & 0 \\
0   & \begin{matrix} b & 1 &0&0\\ 0 & b&0&0 \\ 0 & 0&b&1 \\ 0 & 0&0&b \end{matrix}
\end{pmatrix}$&$\begin{pmatrix}
A_2 & 0 \\
0   & \begin{matrix} b & 1 &1&0\\ 0 & b&0&1 \\ 0 & 0&b&1 \\ 0 & 0&0&b \end{matrix}
\end{pmatrix}$&&
\\ \hline
 &  $[2][1]^2[1]^2$ &
$\begin{pmatrix}
A_2 & 0 \\
0   & \begin{matrix} b & 0 &0&0\\ 0 & b&0&0 \\ 0 & 0&c&0 \\ 0 & 0&0&c \end{matrix}
\end{pmatrix}$ &$\begin{pmatrix}
A_2 & 0 \\
0   & \begin{matrix} b & 1 &0&0\\ 0 & b&0&0 \\ 0 & 0&c&0 \\ 0 & 0&0&c \end{matrix}
\end{pmatrix}$
 &$\begin{pmatrix}
A_2 & 0 \\
0   & \begin{matrix} b & 1 &0&0\\ 0 & b&0&0 \\ 0 & 0&c&1 \\ 0 & 0&0&c\end{matrix}
\end{pmatrix}$&
\\ \hline
 &  $[2][1]^2[1][1]$ &
$\begin{pmatrix}
A_2 & 0 \\
0   & \begin{matrix} b & 0 &0&0\\ 0 & b&0&0 \\ 0 & 0&c&0 \\ 0 & 0&0&d \end{matrix}
\end{pmatrix}$ &$\begin{pmatrix}
A_2 & 0 \\
0   & \begin{matrix} b & 1 &0&0\\ 0 & b&0&0 \\ 0 & 0&c&0 \\ 0 & 0&0&d \end{matrix}
\end{pmatrix}$
 &&
\\ \hline
 &  $[2][1][1][1][1]$ &
$\begin{pmatrix}
A_2 & 0 \\
0   & \begin{matrix} b & 0 &0&0\\ 0 & c&0&0 \\ 0 & 0&d&0 \\ 0 & 0&0&e\end{matrix}
\end{pmatrix}$ &
 &&
\\ \hline
 VIII&  $[1]^2[1]^4$ &
$\begin{pmatrix}
\begin{matrix} a & 0 \\ 0 & a\end{matrix} & 0 \\
0   & \begin{matrix} a & 0 &0&0\\ 0 & b&0&0 \\ 0 & 0&b&0 \\ 0 & 0&0&b \end{matrix}
\end{pmatrix}$ &$\begin{pmatrix}
\begin{matrix} a & 0 \\ 0 & a\end{matrix}  & 0 \\
0   & \begin{matrix} b & 1 &0&0\\ 0 & b&0&0 \\ 0 & 0&b&0 \\ 0 & 0&0&b \end{matrix}
\end{pmatrix}$
 &$\begin{pmatrix}
\begin{matrix} a & 0 \\ 0 & a\end{matrix}  & 0 \\
0   & \begin{matrix} b & 1 &0&0\\ 0 & b&0&0 \\ 0 & 0&b&1 \\ 0 & 0&0&b \end{matrix}
\end{pmatrix}$ &\\ &&$\begin{pmatrix} 
\begin{matrix} a & 0 \\ 0 & a\end{matrix}  & 0 \\
0   & \begin{matrix} b & 1 &1&0\\ 0 & b&0&1 \\ 0 & 0&b&1 \\ 0 & 0&0&b \end{matrix}
\end{pmatrix}$
&
$\begin{pmatrix}
\begin{matrix} a & 1 \\ 0 & a\end{matrix} & 0 \\
0   & \begin{matrix} b & 0 &0&0\\ 0 & b&0&0 \\ 0 & 0&b&0 \\ 0 & 0&0&b \end{matrix}
\end{pmatrix}$ &$\begin{pmatrix}
\begin{matrix} a & 1 \\ 0 & a\end{matrix}  & 0 \\
0   & \begin{matrix} b & 1 &0&0\\ 0 & b&0&0 \\ 0 & 0&b&0 \\ 0 & 0&0&b \end{matrix}
\end{pmatrix}$ & \\ 
&  
 &$\begin{pmatrix}
\begin{matrix} a & 1 \\ 0 & a\end{matrix}  & 0 \\
0   & \begin{matrix} b & 1 &0&0\\ 0 & b&0&0 \\ 0 & 0&b&1 \\ 0 & 0&0&b \end{matrix}
\end{pmatrix}$&$\begin{pmatrix}
\begin{matrix} a & 1 \\ 0 & a\end{matrix}  & 0 \\
0   & \begin{matrix} b & 1 &1&0\\ 0 & b&0&1 \\ 0 & 0&b&1 \\ 0 & 0&0&b \end{matrix}
\end{pmatrix}$&&
\\ \hline
 &  $[1][1][1][1][1][1]$ &
$\begin{pmatrix}
 a & 0&0&0&0&0 \\ 0 & b & 0 &0&0&0\\
0 &0  &  c & 0 &0&0\\ 0 &0  &0 & d&0&0 \\ 0 &0  &0 & 0&e&0 \\0 &0  & 0 & 0&0&f
\end{pmatrix}$ &
 &&
 \\ \hline
\end{tabular}
\end{center}
\newpage

\begin{center}
\fontsize{6pt}{6pt}\selectfont
\renewcommand{\arraystretch}{1.4}
\begin{tabular}{|c|c|cccc|}
\hline
\textbf{Type} & \textbf{Char. pol.} & \textbf{Matrix} & && \\ \hline
 &  $[1]^2[1]^2[1]^2$ &
$\begin{pmatrix}
 a & 0 &0&0&0&0\\ 0 & a&0&0&0&0 \\ 
 0&0&b & 0 &0&0\\ 0&0&0 & b&0&0 \\ 0&0&0 & 0&c&0 \\ 0&0&0 & 0&0&c 
\end{pmatrix}$ &$\begin{pmatrix}
 a & 1 &0&0&0&0\\ 0 & a&0&0&0&0 \\ 
 0&0&b & 0 &0&0\\ 0&0&0 & b&0&0 \\ 0&0&0 & 0&c&0 \\ 0&0&0 & 0&0&c 
\end{pmatrix}$
 &&
\\ 
 &   &
$\begin{pmatrix}
 a & 1 &0&0&0&0\\ 0 & a&0&0&0&0 \\ 
 0&0&b & 1 &0&0\\ 0&0&0 & b&0&0 \\ 0&0&0 & 0&c&0 \\ 0&0&0 & 0&0&c 
\end{pmatrix}$ &$\begin{pmatrix}
 a & 1 &0&0&0&0\\ 0 & a&0&0&0&0 \\ 
 0&0&b & 1 &0&0\\ 0&0&0 & b&0&0 \\ 0&0&0 & 0&c&1 \\ 0&0&0 & 0&0&c 
\end{pmatrix}$
 &&
\\ \hline
 &  $[1]^2[1]^2[1][1]$ &
$\begin{pmatrix}
 a & 0 &0&0&0&0\\ 0 & a&0&0&0&0 \\ 
 0&0&b & 0 &0&0\\ 0&0&0 & b&0&0 \\ 0&0&0 & 0&c&0 \\ 0&0&0 & 0&0&d 
\end{pmatrix}$ &$\begin{pmatrix}
 a & 1 &0&0&0&0\\ 0 & a&0&0&0&0 \\ 
 0&0&b & 0 &0&0\\ 0&0&0 & b&0&0 \\ 0&0&0 & 0&c&0 \\ 0&0&0 & 0&0&d 
\end{pmatrix}$
 &
$\begin{pmatrix}
 a & 1 &0&0&0&0\\ 0 & a&0&0&0&0 \\ 
 0&0&b & 1 &0&0\\ 0&0&0 & b&0&0 \\ 0&0&0 & 0&c&0 \\ 0&0&0 & 0&0&d
\end{pmatrix}$ &
 
\\ \hline

 &  $[1]^2[1][1][1][1]$ &
$\begin{pmatrix}
\begin{matrix} a & 0 \\ 0 & a\end{matrix}  & 0 \\
0   & \begin{matrix} b & 0 &0&0\\ 0 & c&0&0 \\ 0 & 0&d&0 \\ 0 & 0&0&e\end{matrix}
\end{pmatrix}$ &$\begin{pmatrix}
\begin{matrix} a & 1 \\ 0 & a\end{matrix}  & 0 \\
0   & \begin{matrix} b & 0 &0&0\\ 0 & c&0&0 \\ 0 & 0&d&0 \\ 0 & 0&0&e\end{matrix}
\end{pmatrix}$ 
 &&
\\ \hline
 &  $[1]^3[1]^3$ &
$\begin{pmatrix}
\begin{matrix} a & 0&0 \\ 0 & a&0\\0 & 0&a
\end{matrix}  & 0 \\
0   & \begin{matrix} b & 0 &0\\ 0 & b&0& \\ 0 & 0&b\end{matrix}
\end{pmatrix}$ &$\begin{pmatrix}
\begin{matrix} a & 1&0 \\ 0 & a&1\\0 & 0&a
\end{matrix}  & 0 \\
0   & \begin{matrix} b & 1 &0\\ 0 & b&1& \\ 0 & 0&b\end{matrix}
\end{pmatrix}$
 &   $\begin{pmatrix}
\begin{matrix} a & 1&0 \\ 0 & a&1\\0 & 0&a
\end{matrix}  & 0 \\
0   & \begin{matrix} b & 0 &0\\ 0 & b&0& \\ 0 & 0&b\end{matrix}
\end{pmatrix}$&
\\ \hline
& $[1]^6$  &
$\begin{pmatrix}
a&1&0&0&0&0\\
0&a&1&0&0&0\\
0&0&a&1&0&0\\
0&0&0&a&1&0\\
0&0&0&0&a&1\\
0&0&0&0&0&a
\end{pmatrix}$ &
$\begin{pmatrix}
a&1&0&0&0&0\\
0&a&1&0&0&0\\
0&0&a&1&0&0\\
0&0&0&a&0&0\\
0&0&0&0&a&1\\
0&0&0&0&0&a
\end{pmatrix}$ &
$\begin{pmatrix}
a&1&0&0&0&0\\
0&a&1&0&0&0\\
0&0&a&1&0&0\\
0&0&0&a&0&0\\
0&0&0&0&a&0\\
0&0&0&0&0&a
\end{pmatrix}$& \\
& &
$\begin{pmatrix}
a&1&0&0&0&0\\
0&a&1&0&0&0\\
0&0&a&0&0&0\\
0&0&0&a&1&0\\
0&0&0&0&a&1\\
0&0&0&0&0&a
\end{pmatrix}$&
$\begin{pmatrix}
a&1&0&0&0&0\\
0&a&0&0&0&0\\
0&0&a&1&0&0\\
0&0&0&a&1&0\\
0&0&0&0&a&1\\
0&0&0&0&0&a
\end{pmatrix}$ &
$\begin{pmatrix}
a&1&0&0&0&0\\
0&a&0&0&0&0\\
0&0&a&1&0&0\\
0&0&0&a&0&0\\
0&0&0&0&a&0\\
0&0&0&0&0&a
\end{pmatrix}$ &\\
&&
$\begin{pmatrix}
a&1&0&0&0&0\\
0&a&0&0&0&0\\
0&0&a&0&0&0\\
0&0&0&a&0&0\\
0&0&0&0&a&0\\
0&0&0&0&0&a
\end{pmatrix}$ &
$\begin{pmatrix}
a&0&0&0&0&0\\
0&a&0&0&0&0\\
0&0&a&0&0&0\\
0&0&0&a&0&0\\
0&0&0&0&a&0\\
0&0&0&0&0&a
\end{pmatrix}$&&\\
\hline

\end{tabular}
\end{center}
\newpage

\newpage

\section{Appendix B}
In the following table, we give the factorization types of the 2-torsion Galois orbits according to the factorization of the characteristic polynomial over $\mathbb{F}_2$
\begin{table}[h!]
\centering
\fontsize{7pt}{7pt}\selectfont
\begin{tabular}{|c|l|l|c|}
\hline
\textbf{Type} & \textbf{Characteristic Polynomial} & \textbf{Matrix}  & \textbf{Orbits}  \\ \hline
I     & $x^6 + x^3 + 1$  &  $
\begin{pmatrix}
0 & 1 & 0 & 0 & 0 & 0 \\
0 & 0 & 1 & 0 & 0 & 0 \\
0 & 0 & 0 & 1 & 0 & 0 \\
0 & 0 & 0 & 0 & 1 & 0 \\
0 & 0 & 0 & 0 & 0 & 1 \\
1 & 0 & 0 & 1 & 0 & 0
\end{pmatrix}
$ & $(9)^7$
\\ \hline
II    & $(x^4 + x^3 + x^2 + x + 1)(x^2 + x + 1)$ &$\begin{pmatrix}
0 & 1 & 0 & 0 & 0 & 0 \\
0 & 0 & 1 & 0 & 0 & 0 \\
0 & 0 & 0 & 1 & 0 & 0 \\
1 & 1 & 1 & 1 & 0 & 0 \\
0 & 0 & 0 & 0 & 0 & 1 \\
0 & 0 & 0 & 0 & 1 & 1
\end{pmatrix}$ &$(3)(5)^3(15)^3$\\ \hline
III   & $(x^4 + x^3 + x^2 + x + 1)( x + 1)^2$ & $\begin{pmatrix}
0 & 1 & 0 & 0 & 0 & 0 \\
0 & 0 & 1 & 0 & 0 & 0 \\
0 & 0 & 0 & 1 & 0 & 0 \\
1 & 1 & 1 & 1 & 0 & 0 \\
0 & 0 & 0 & 0 & 1 & 0 \\
0 & 0 & 0 & 0 & 0 & 1
\end{pmatrix}
$ & $(1)^3(5)^{12}$\\ && $\begin{pmatrix}
0 & 1 & 0 & 0 & 0 & 0 \\
0 & 0 & 1 & 0 & 0 & 0 \\
0 & 0 & 0 & 1 & 0 & 0 \\
1 & 1 & 1 & 1 & 0 & 0 \\
0 & 0 & 0 & 0 & 1 & 1 \\
0 & 0 & 0 & 0 & 0 & 1
\end{pmatrix}
$ & $(1)(2)(5)^{6}(10)^3$ \\ \hline
IV    & $(x^3 + x^2 + 1)(x^3 + x + 1)$ &$\begin{pmatrix}
0 & 1 & 0 & 0 & 0 & 0 \\
0 & 0 & 1 & 0 & 0 & 0 \\
1 & 1 & 0 & 0 & 0 & 0 \\
0 & 0 & 0 & 0 & 1 & 0 \\
0 & 0 & 0 & 0 & 0 & 1 \\
0 & 0 & 0 & 1 & 0 & 1
\end{pmatrix}$ &$(7)^9$\\ \hline
V     & $(x^2 + x + 1)^3$ &$\begin{pmatrix}
0 & 1 & 0 & 0 & 0 & 0 \\
1 & 1 & 0 & 0 & 0 & 0 \\
0 & 0 & 0 & 1 & 0 & 0 \\
0 & 0 & 1 & 1 & 0 & 0 \\
0 & 0 & 0 & 0 & 0 & 1 \\
0 & 0 & 0 & 0 & 1 & 1
\end{pmatrix}$& $(3)^{21}$\\ &  &$\begin{pmatrix}
0 & 1 & 1 & 0 & 0 & 0 \\
1 & 1 & 0 & 1 & 0 & 0 \\
0 & 0 & 0 & 1 & 1 & 0 \\
0 & 0 & 1 & 1 & 0 & 1 \\
0 & 0 & 0 & 0 & 0 & 1 \\
0 & 0 & 0 & 0 & 1 & 1
\end{pmatrix}$& $(3)^{1}(6)^2(12)^4$\\ &  &$\begin{pmatrix}
0 & 1 & 0 & 0 & 0 & 0 \\
1 & 1 & 0 & 0 & 0 & 0 \\
0 & 0 & 0 & 1 & 1 & 0 \\
0 & 0 & 1 & 1 & 0 & 1 \\
0 & 0 & 0 & 0 & 0 & 1 \\
0 & 0 & 0 & 0 & 1 & 1
\end{pmatrix}$& $(3)^{5}6^8$\\ \hline
\end{tabular}
\end{table}

\newpage

\begin{table}[h!]
\centering
\fontsize{8pt}{8pt}\selectfont
\begin{tabular}{|c|l|l|c|}
\hline
VI    & $(x^2 + x + 1)^2(x+1)^2$ & 
 $\begin{pmatrix}
0 & 1 & 0 & 0 & 0 & 0 \\
1 & 1 & 0 & 0 & 0 & 0 \\
0 & 0 & 0 & 1 & 0 & 0 \\
0 & 0 & 1 & 1 & 0 & 0 \\
0 & 0 & 0 & 0 & 1 & 0 \\
0 & 0 & 0 & 0 & 0 & 1
\end{pmatrix}$ & $(1)^{3}(3)^{20}$\\ 
&  & 
 $\begin{pmatrix}
0 & 1 & 0 & 0 & 0 & 0 \\
1 & 1 & 0 & 0 & 0 & 0 \\
0 & 0 & 0 & 1 & 0 & 0 \\
0 & 0 & 1 & 1 & 0 & 0 \\
0 & 0 & 0 & 0 & 1 & 1 \\
0 & 0 & 0 & 0 & 0 & 1
\end{pmatrix}$ & $(1)(2)(3)^{10}(6)^5$\\ 
&  & 
 $\begin{pmatrix}
0 & 1 & 1 & 0 & 0 & 0 \\
1 & 1 & 0 & 1 & 0 & 0 \\
0 & 0 & 0 & 1 & 0 & 0 \\
0 & 0 & 1 & 1 & 0 & 0 \\
0 & 0 & 0 & 0 & 1 & 0 \\
0 & 0 & 0 & 0 & 0 & 1
\end{pmatrix}$& $(1)^3(3)^{4}(6)^8$\\
&  & 
 $\begin{pmatrix}
0 & 1 & 1 & 0 & 0 & 0 \\
1 & 1 & 0 & 1 & 0 & 0 \\
0 & 0 & 0 & 1 & 0 & 0 \\
0 & 0 & 1 & 1 & 0 & 0 \\
0 & 0 & 0 & 0 & 1 & 1 \\
0 & 0 & 0 & 0 & 0 & 1
\end{pmatrix}$& $(1)(2)(3)^{2}(6)^9$\\ \hline
VII   & $(x^2 + x + 1)(x+1)^4$ & $\begin{pmatrix}
0&1&0&0&0&0\\
1&1&0&0&0&0\\
0&0&1&0&0&0\\
0&0&0&1&0&0\\
0&0&0&0&1&0\\
0&0&0&0&0&1
\end{pmatrix}$ & $(1)^{15}(3)^{16}$\\ && $\begin{pmatrix}
0&1&0&0&0&0\\
1&1&0&0&0&0\\
0&0&1&1&0&0\\
0&0&0&1&0&0\\
0&0&0&0&1&0\\
0&0&0&0&0&1
\end{pmatrix}$ & $(1)^{7}(2)^4(3)^{8}(6)^4$\\ 
&& $\begin{pmatrix}
0&1&0&0&0&0\\
1&1&0&0&0&0\\
0&0&1&1&0&0\\
0&0&0&1&0&0\\
0&0&0&0&1&1\\
0&0&0&0&0&1
\end{pmatrix}$&$(1)^{3}(2)^6(3)^{4}(6)^6$\\ 
&& $\begin{pmatrix}
0&1&0&0&0&0\\
1&1&0&0&0&0\\
0&0&1&1&0&0\\
0&0&0&1&1&0\\
0&0&0&0&1&1\\
0&0&0&0&0&1
\end{pmatrix}$ &$(1)(2)(3)^{2}(4)^3(6)$\\ 
\hline
\end{tabular}
\end{table}

\newpage

\begin{table}[h!]
\centering
\fontsize{8pt}{8pt}\selectfont
\begin{tabular}{|c|l|l|c|}
\hline
VIII & $(x+1)^6$ &
$\begin{pmatrix}
1&1&0&0&0&0\\
0&1&1&0&0&0\\
0&0&1&1&0&0\\
0&0&0&1&1&0\\
0&0&0&0&1&1\\
0&0&0&0&0&1
\end{pmatrix}$ & $(1)(2)(4)^3(8)^6$\\
&&
$\begin{pmatrix}
1&1&0&0&0&0\\
0&1&1&0&0&0\\
0&0&1&1&0&0\\
0&0&0&1&0&0\\
0&0&0&0&1&1\\
0&0&0&0&0&1
\end{pmatrix}$ & $(1)^3(2)^6(4)^{12}$\\
& &
$\begin{pmatrix}
1&1&0&0&0&0\\
0&1&1&0&0&0\\
0&0&1&1&0&0\\
0&0&0&1&0&0\\
0&0&0&0&1&0\\
0&0&0&0&0&1
\end{pmatrix}$ & $(1)^7(2)^4(4)^{12}$\\
& &
$\begin{pmatrix}
1&1&0&0&0&0\\
0&1&1&0&0&0\\
0&0&1&0&0&0\\
0&0&0&1&1&0\\
0&0&0&0&1&1\\
0&0&0&0&0&1
\end{pmatrix}$ &$(1)^3(2)^6(4)^{12}$\\
& &
$\begin{pmatrix}
1&1&0&0&0&0\\
0&1&0&0&0&0\\
0&0&1&1&0&0\\
0&0&0&1&1&0\\
0&0&0&0&1&1\\
0&0&0&0&0&1
\end{pmatrix}$ & $(1)^7(2)^{28}$\\
& &
$\begin{pmatrix}
1&1&0&0&0&0\\
0&1&0&0&0&0\\
0&0&1&1&0&0\\
0&0&0&1&0&0\\
0&0&0&0&1&0\\
0&0&0&0&0&1
\end{pmatrix}$ & $(1)^{15}(2)^{24}$\\
& &
$\begin{pmatrix}
1&1&0&0&0&0\\
0&1&0&0&0&0\\
0&0&1&0&0&0\\
0&0&0&1&0&0\\
0&0&0&0&1&0\\
0&0&0&0&0&1
\end{pmatrix}$&$(1)^{31}(2)^{16}$\\
& &
$\begin{pmatrix}
1&0&0&0&0&0\\
0&1&0&0&0&0\\
0&0&1&0&0&0\\
0&0&0&1&0&0\\
0&0&0&0&1&0\\
0&0&0&0&0&1
\end{pmatrix}$&$(1)^{63}$\\
\hline

\end{tabular}
\caption{Characteristic polynomials for $\ell = 2$}
\end{table}

\end{document}